\documentclass[11pt]{amsart}

\usepackage{amsmath,amssymb,amsthm}
\usepackage{amsmath,amssymb,amsthm,amscd}
\usepackage[frame,cmtip,arrow,matrix,line,graph,curve]{xy}
\usepackage{graphpap, color}
\usepackage[mathscr]{eucal}
\usepackage{color}
\numberwithin{equation}{section}
\newtheorem{prop}{Proposition}
\newtheorem{theo}[prop]{Theorem}
\newtheorem{lemm}[prop]{Lemma}
\newtheorem{coro}[prop]{Corollary}
\newtheorem{rema}[prop]{Remark}

\newtheorem*{Aleks}{Lemma (Aleksandrov)}
\newtheorem*{KS}{Theorem (Krylov--Safonov)}
\theoremstyle{definition}

\newtheorem*{pflemC0}{Proof of Lemma~\ref{le:C0}}
\newtheorem*{pflemC2}{Proof of Lemma~\ref{le:C2}}

\newtheorem*{ack}{Acknowledgment}

\newcommand{\td}{\tilde}
\newcommand{\p}{\partial}

\newcommand{\fr}{\frac{\sqrt{-1}}{2}}
\newcommand{\ppr}{(\sqrt{-1}/2) \partial \bar{\partial}}

\newcommand{\dm}{\textup{diam}}

\def\lab{\label}

\begin{document}
\title[Form--type equations]{Form-type Calabi-Yau equations on K\"ahler manifolds of nonnegative orthogonal bisectional curvature}
\author{Jixiang Fu}
\address{Institute of Mathematics\\ Fudan University \\ Shanghai
200433, China} \email{majxfu@fudan.edu.cn}
\author{Zhizhang Wang}
\address{Institute of Mathematics\\ Fudan University \\ Shanghai
200433, China} \email{youxiang163wang@163.com}
\author{Damin Wu}
\address{Department of Mathematics \\
         The Ohio State University \\
         1179 University Drive, Newark, OH 43055, U.S.A.}
\email{dwu@math.ohio-state.edu}

\begin{abstract}
  In this paper we prove the existence and uniqueness of the form-type Calabi-Yau
   equation on K\"ahler manifolds of nonnegative orthogonal bisectional curvature.
\end{abstract}

\maketitle

\section{Introduction}
In the previous paper \cite{FuWangWu}, we introduced the form--type Calabi--Yau equation on
a compact complex $n$-dimensional manifold with a balanced metric and with a non-vanishing holomorphic
$n$-form $\Omega$. A balanced metric $\omega$ on $X$ is a hermitian metric such that $d\omega^{n-1}=0$. Given a
 balanced metric $\omega_0$ on $X$, let us denote by $\mathcal P(\omega_0)$ the set of all smooth real
  $(n-2,n-2)$--forms $\psi$ such that
$\omega_0^{n-1}+\frac{\sqrt{-1}}{2}\partial\bar\partial\psi>0$ on $X$. Then, for each $\varphi\in\mathcal P(\omega_0)$, there exists a balanced metric, which we denote by $\omega_{\varphi}$, such that $\omega_{\varphi}^{n-1}=\omega_0^{n-1}+\frac{\sqrt{-1}}{2}\partial\bar\partial\varphi$. We say that such a metric $\omega_\varphi$ is in the balanced class of $\omega_0$.
  Our aim is to find  a balanced metric $\omega_\varphi$ in the balanced class of $\omega_0$
such that
\begin{equation}\lab{101}
\|\Omega\|_{\omega_\varphi}=\textup{a constant}\ C_0>0.
\end{equation}
The geometric meaning of such a metric is that its Ricci curvatures
of the hermitian connection and the spin connection are zero. On the
other hand, the direct non--K\"ahler analogue of the Calabi
conjecture has recently been solved by Tosatti--Weinkove~\cite{TW}
(see also \cite{GL}, and the references in \cite{TW, GL}.). In general their solutions provides hermitian
Ricci-flat metrics which are not balanced.

As in the K\"ahler case, equation (\ref{101}) can be reformulated in the following form
\begin{equation} \label{eq:FT0}
   \frac{\omega_\varphi^n}{\omega_0^n} = e^f \frac{\int_X \omega_\varphi^n }{\int_X \omega_0^n},
\end{equation}
where $f \in C^{\infty}(X)$ is given and satisfies  the compatibility condition:
\begin{equation}\label{eq:cond}
   \int_X e^f \omega_0^n = \int_X \omega_0^n.
\end{equation}
We would like to find a solution $\varphi\in \mathcal{P}(\omega_0)$.
The equation \eqref{eq:FT0} is called a \emph{form-type Calabi--Yau
equation}, a reminiscent of the classic function type Calabi--Yau
equation.  We have constructed solutions for (\ref{101}) when $X$ is
a complex torus~\cite{FuWangWu}. A natural approach to solve
\eqref{eq:FT0} is to use the continuity method. The openness and
uniqueness were discussed in the previous work~\cite{FuWangWu}. We
do not know whether there is a geometric obstruction for solving
\eqref{eq:FT0} in general.

Equation \eqref{eq:FT0} is still meaningful on a compact complex manifold with a balanced metric, whose canonical line bundle is not holomorphically trivial. Geometrically, solving \eqref{eq:FT0} allows us to solve the problem of prescribed volume form on $X$, in the balanced  class of each balanced metric on $X$. Namely, given any positive $(n,n)$--form $W$ on $X$ and a balanced metric $\omega_0$, we let
\[
   e^f = \left(\frac{W}{\omega_0^n}\right) \frac{\int_X \omega_0^n}{\int_X W};
\]
then by solving \eqref{eq:FT0} we are able to find a metric $\omega_\varphi$ in the balanced class of $\omega_0$ such that $\omega_\varphi^n$ is equal to $W$, up to a constant rescaling.

It seems to us very hard to understand equation \eqref{eq:FT0} in general. In this paper, we want to look for solutions in a subset within the balanced class of a given balanced metric. The idea is, in some sense, to transfer the form-type Calabi--Yau equation to a function type equation. 

So in the following we let $(X, \eta)$ be an $n$-dimensional K\"ahler manifold, $n \ge 2$, and $\omega_0$ be a balanced metric on $X$. We let
\[
   \mathcal{P}_{\eta}(\omega_0) = \{ v \in C^{\infty}(X)\mid \omega_0^{n-1} + \ppr v \wedge \eta^{n-2}  >0 \quad \textup{on $X$}\}.
\]
For each $v \in \mathcal{P}_{\eta}(\omega_0)$, we denote by $\omega_v$ the positive $(1,1)$--form on $X$ such that
\[
   \omega_v^{n - 1} = \omega_0^{n-1} + \ppr v \wedge \eta^{n-2} \qquad \textup{on $X$}.
\]
Then we consider the equation
\begin{equation} \label{eq:FT}
   \frac{\omega_u^n}{\omega_0^n} = e^f \frac{\int_X \omega_u^n }{\int_X \omega_0^n},
\end{equation}
where $f \in C^{\infty}(X)$ is given and satisfies  the compatibility condition \eqref{eq:cond}.

Obviously, the function-type equation~\eqref{eq:FT} is a special
case of the form-type equation~\eqref{eq:FT0}. However, an important
observation is that, solving \eqref{eq:FT} will enable one to find
all the solutions to \eqref{eq:FT0}, in the balanced class of a
given balanced metric (see Remark~\ref{re:idea}).

In this paper, we are able to solve \eqref{eq:FT}, under the assumption that the K\"ahler metric $\eta$ has \emph{nonnegative orthogonal bisectional curvature}; that is,
for any orthonormal tangent frame $\{ e_1, \ldots , e_n\}$ at any $x\in M$, the curvature tensor of $\eta$ satisfies that
\begin{equation} \label{eq:nonbis}
   R_{i\bar{i}j\bar{j}} \equiv R (e_i, \bar{e}_i, e_j, \bar{e}_j) \geq 0, \qquad \textup{for all $1 \le i,j \le n$ and $i \ne j$}.
\end{equation}
We remark that nonnegativity of the orthogonal bisectional curvature
is weaker than nonnegativity of the bisectional curvature. In fact,
the former condition are satisfied by not only complex projective
spaces and the Hermitian symmetric spaces, but also some compact
K\"ahler manifolds of dimension $\ge 2$ whose holomorphic sectional
curvature is strictly negative somewhere. We refer the reader to the
recent work Gu--Zhang~\cite{GZ} for the study of nonnegative
orthogonal bisectional curvature, which generalizes the earlier work of Mok
\cite{M} and Siu--Yau \cite{SY}.

Our main result is as follows:
\begin{theo} \label{th:main}
  Let $(X,\eta)$ be a compact K\"ahler manifold of nonnegative orthogonal bisectional curvature, and $\omega_0$ be a balanced metric on $X$. Then, for any smooth function $f$ on $X$, equation~\eqref{eq:FT} admits a solution $u \in \mathcal{P}_{\eta}(\omega_0)$, which is unique up to a constant.
\end{theo}

\begin{rema} \label{re:idea}
Recall that our aim is to find all solutions of
equation \eqref{eq:FT0} in the balanced class of a given balance
metric $\omega_0$ for a given $f\in C^\infty(X)$ satisfying the
compatibility condition \eqref{eq:cond}. By Theorem~\ref{th:main},
we achieve this goal on a K\"ahler manifold $X$ of nonnegative
orthogonal bisectional curvature. \textup{(}In particular, the
form-type equation on a complex torus is completely settled.\textup{)}

Indeed, for any $\varphi \in \mathcal{P}(\omega_0)$ and any $f\in
C^{\infty}(X)$ satisfying \eqref{eq:cond}, we claim that there
exists a unique smooth function $u_{\varphi}$ up to a constant such
that $(\varphi + u_{\varphi} \wedge \eta^{n-2})$ is in
$\mathcal{P}(\omega_0)$ and
solves \eqref{eq:FT0}; namely, if we denote by $\omega_{u_\varphi}$
the positive $(1,1)$--form such that
\[
    \omega_{u_{\varphi}}^{n-1} = \omega_0^{n-1}+\frac{\sqrt{-1}}2\partial\bar\partial\varphi
+\frac{\sqrt{-1}}2\partial\bar\partial u_\varphi\wedge\eta^{n-2} >0,
\]
then,
\begin{equation} \label{eq:102}
   \frac{\omega_{u_\varphi}^n}{\omega_0^n}=e^{f}\frac{\int_{X}\omega_{u_\varphi}^n}{\int_{X}\omega_0^n}.
\end{equation}
To see this, we define a function $f_{\varphi} \in C^{\infty}(X)$ by
\[
        e^{f_\varphi}=e^f\frac{\omega_0^n}{\omega_{\varphi}^n}\frac{\int_X\omega_{\varphi}^n}{\int_X\omega_0^n}.
\]
Then, $f_{\varphi}$ satisfies
\begin{equation*}
    \int_{X}e^{f_\varphi}\omega_\varphi^n=
    \frac{\int_{X}e^f\omega_0^n\int_X\omega_{\varphi}^n}{\int_X\omega_0^n}
        =\int_{X}\omega_\varphi^n.
\end{equation*}
Applying Theorem~\ref{th:main} to $\omega_{\varphi}$ and $f_{\varphi}$ yields that, there exists
a unique solution $u_\varphi\in\mathcal P_\eta(\omega_\varphi)$ satisfying
that
\begin{equation*}
\frac{\omega_{u_\varphi}^n}{\omega_\varphi^n}=e^{f_\varphi}\frac{\int_{X}\omega_{u_\varphi}^n}{\int_{X}\omega_\varphi^n}.
\end{equation*}
Clearly, this is equivalent to \eqref{eq:102}. The claim is proved.
Therefore, in this way, we can find all solutions
\textup{(}which are infinitely many\textup{)} to equation \eqref{eq:FT0} in the
balanced class of a given balanced metric on $X$ of nonnegative
orthogonal bisectional curvature.

So the idea used in this paper, which is to transfer from the form-type
Calabi--Yau equation to a function-type equation, may be useful.
Later we will establish the theorem 1 on any compact K\"ahler
manifold.
\end{rema}

We employ the continuity method to prove Theorem~\ref{th:main}. In Section~\ref{se:C2}, we establish an \emph{a priori} $C^2$ estimate for the solution $u$. This is the place where we need the curvature condition. The $C^2$ estimate enables us to obtain a general \emph{a priori} $C^0$ estimate, via the classic Moser's iteration. 
This is the content of Section~\ref{se:C0}. We then adapt the
Evans--Krylov theory to our form--type equation, and obtain in
Section~\ref{se:Hld} the H\"older estimates for second derivatives.
The openness is covered by Theorem 3 in our previous
paper~\cite{FuWangWu}. For readers' convenience, we briefly indicate
the argument in the last section, Section~\ref{se:open}. The
uniqueness is also proved in Section~\ref{se:open}.

\begin{ack}
The authors would like to thank Professor S.-T.~Yau for helpful
discussion. Part of the work was done while the third named author
was visiting Fudan University, he would like to thank their  warm
hospitality. Fu is supported in part by NSFC grants 10831008 and
11025103.
\end{ack}

\section{$C^2$ estimates for Form--type equations} \label{se:C2}

In this section, we would like to establish the following estimate:
\begin{lemm} \label{le:C2}
Given $F \in C^2(X)$, let $u \in C^4(X)$ satisfy that 
\[
   \omega_0^{n-1} + \ppr u \wedge \eta^{n-2} > 0 \qquad \textup{on $X$},
\]
and that
\begin{eqnarray} \label{eq:Cf}
\det [\omega_0^{n-1}+\ppr u\wedge \eta^{n-2}]= e^{F} \det\omega_0^{n-1}.
\end{eqnarray}
Then, we have
\begin{equation} \label{eq:Laplu}
    \Delta_{\eta} u \le C + C( u - \inf_X u) \qquad \textup{on $X$},
\end{equation}
and
\[
   \sup_X |\omega_0^{n-1} + \partial \bar{\partial} u \wedge \eta^{n-2}|_{\eta} \le C + \Big(\sup_X u - \inf_X u \Big).
\]
Here $\Delta_{\eta} v =\sum \eta^{i\bar j} v_{i\bar j}$ denotes the Laplacian of a function $v$ with respect to $\eta$, and $C>0$ is a constant depending only
on $\inf_X (\Delta_{\eta}F)$, $\sup_X F$, $\eta$, $n$, and $\omega_0$.
\end{lemm}

Here are some conventions:
      For an $(n-1,n-1)$--form $\Theta$, we denote
\begin{equation*} \label{eq:index}
  \begin{split}
        \Theta & = \Big(\frac{\sqrt{-1}}{2} \Big)^{n-1} (n-1)! \\
            & \quad \cdot \sum_{p,q}s(p,q)\Theta_{p\bar{q}}dz^1\wedge d\bar{z}^1 \cdots
            \wedge\widehat{dz^p} \wedge d \bar{z}^p \wedge\cdots \wedge
            d\bar{z}^q \wedge \widehat{d\bar{z}^q}\wedge\cdots \wedge d z^n \wedge
            d\bar{z}^n,
    \end{split}
\end{equation*}
in which 
\begin{equation} \label{eq:sign}
   s(p,q) =
    \begin{cases}
      - 1,  & \textup{if $p > q$}; \\
      1, & \textup{if $p \le q$}.
    \end{cases}
 \end{equation}
 Here we introduce the sign function $s$ so that, 
 \[
   \begin{split}
    & d z^p \wedge d\bar{z}^q \wedge s(p,q) dz^1\wedge d\bar{z}^1 \cdots
            \wedge\widehat{dz^p} \wedge d \bar{z}^p \wedge\cdots \wedge
            d\bar{z}^q \wedge \widehat{d\bar{z}^q}\wedge\cdots \wedge d z^n \wedge
            d\bar{z}^n \\
        & = dz^1 \wedge d\bar{z}^1 \wedge \cdots \wedge d z^n \wedge d\bar{z}^n, \qquad \textup{for all $1 \le p, q \le n$}.
    \end{split}
 \]

We denote
\[
   \det \Theta = \det (\Theta_{p\bar{q}}).
\]
 If the matrix $(\Theta_{p\bar{q}})$ is invertible, we denote by $(\Theta^{p\bar{q}})$ the transposed inverse of $(\Theta_{p\bar{q}})$, i.e.,
\[
    \sum_l \Theta_{i\bar{l}} \Theta^{j\bar{l}} = \delta_{ij}.
\]
Note that, for a positive $(1,1)$--form $\omega$ given by
\[
   \omega = \fr \sum_{i,j=1}^n g_{i\bar{j}} dz_i \wedge d\bar{z}_j,
\]
we have
\[
   \omega^n = \Big(\frac{\sqrt{-1}}{2} \Big)^{n} n! \det (g_{i\bar{j}}) dz^1\wedge d\bar{z}_1 \wedge\cdots \wedge dz^n \wedge d\bar{z}^n,
\]
and by our convention,
\[
   (\omega^{n-1})_{i\bar{j}} = \det(g_{i\bar{j}}) g^{i\bar{j}}.
\]
It follows that
\begin{equation} \label{eq:n-1}
   \det (\omega^{n-1}) = \det (g_{i\bar{j}})^{n-1},
\end{equation}
and
\[
   (\omega^{n-1})^{i\bar{j}} = \frac{g_{i\bar{j}}}{\det (g_{i\bar{j}})}.
\]

In the following, the subscripts such as ``$,p$'' stand for the ordinary local derivatives; for example,
\begin{equation} \label{eq:nt}
    \eta_{i\bar{j},k} = \frac{\partial \eta_{i\bar{j}}}{\partial z^k}, \quad \eta_{i\bar{j},l\bar{m}} = \frac{\partial^2 \eta_{i\bar{j}}}{\partial z^l \partial \bar{z}^m}.
\end{equation}
For a function $h$ we can omit the comma: $h_l = h_{,l}$, $h_{l\bar{m}} = h_{,l\bar{m}}$, etc.
Unless otherwise indicated, all the summations below range from $1$ to $n$. We remark that, under the convention, equation~\eqref{eq:FT} can be rewritten as
\[
   \frac{\det[\omega_0^{n-1} + \ppr u \wedge \eta^{n-2}]}{\det \omega_0^{n-1}} = e^{(n-1)f} \left(\frac{\int_X \omega_u^n }{\int_X \omega_0^n} \right)^{n-1},
\]
which is convenient for deriving the estimates.

\begin{pflemC2}
Let
\begin{equation*}
    \Psi_u=\Psi+\ppr u\wedge\eta^{n-2}, \ \ \ \text{where}\ \ \Psi=\omega_0^{n-1}.
\end{equation*}
Let
\begin{align*}
  \phi
  & = \frac{\sum_{i,j} \eta_{i\bar{j}} (\Psi_u)_{i\bar{j}}}{\det\eta} - A u,
\end{align*}
where $A >0$ is a large constant to be determined. Using wedge products, the function $\phi$ can also be written as
\begin{equation} \label{eq:Lphi}
  \begin{split}
  \phi
    & = \frac{n \eta \wedge \Psi_u}{\eta^n} - Au \\
    & = (h + \Delta_{\eta} u) - A u, \quad \textup{where $h = \frac{n \eta \wedge \omega_0^{n-1}}{\eta^n}$}.
  \end{split}
\end{equation}
Consider
the operator
\begin{equation*}
L\phi=(n-1)\sum_{k,l} \Psi_{u}^{k\bar{l}}\big( \frac{\sqrt{-1}}{2}\partial \bar{\partial}\phi\wedge\eta^{n-2} \big)_{k\bar{l}}.
\end{equation*}

Suppose that $\phi$ attains its maximum at some point $P$ in $X$. We choose a normal coordinates, such that at $P$,
$\eta_{i\bar{j}}=\delta_{ij}$ and $d\eta_{i\bar{j}}=0$. Then,
we rotate the axes so that at $P$
 we have $(\Psi_u)_{p\bar{q}}=\delta_{pq}(\Psi_u)_{p\bar{p}}$.
 Thus, for any smooth function $v$ on $X$, we have at $P$ that
 \begin{equation} \label{eq:cof}
    (n-1)\big(\frac{\sqrt{-1}}{2}\partial \bar{\partial} v \wedge \eta^{n-2}\big)_{i\bar{j}} = \delta_{ij}\sum_{p \ne i} v_{p\bar{p}} + (1 - \delta_{ij})v_{j\bar{i}}.
 \end{equation}
 By \eqref{eq:cof} we obtain that
 \begin{align}
    (\Psi_u)_{i\bar{i}} & = \Psi_{i\bar{i}} + \frac{1}{n-1} \sum_{q \ne i} u_{q\bar{q}}, \label{eq:diagPsiu} \\
    (\Psi_u)_{i\bar{j}} & = \Psi_{i\bar{j}} + \frac{u_{j\bar{i}}}{n - 1}  = 0,  \qquad \textup{for all $i \ne j$}. \label{eq:odPsiu}
 \end{align}
 It follows that
 \begin{equation} \label{eq:trace}
    \sum_{i=1}^n (\Psi_u)_{i\bar{i}}
    =  \sum_{i=1}^n \Psi_{i\bar{i}} +  \sum_{i=1}^n u_{i\bar{i}} = h + \Delta_{\eta} u.
 \end{equation}
 Furthermore, we have
 \begin{equation} \label{eq:nmdPsi}
    (\Psi_u)_{i\bar{j},p} = \Psi_{i\bar{j},p} + \frac{\delta_{ij}}{n-1} \sum_{q \ne i} u_{q\bar{q}p} + \frac{1 - \delta_{ij}}{n - 1} u_{j\bar{i}p},
 \end{equation}
 and
 \begin{equation*} 
   \begin{split}
    (\Psi_u)_{i\bar{i},p\bar{p}}
    &  = \Psi_{i\bar{i},p\bar{p}} + \frac{1}{n-1}\sum_{k \ne i} u_{k\bar{k}p\bar{p}}  + \frac{1}{n - 1}\sum_{k \ne i} u_{k\bar{k}}\Big(\sum_{j \ne k, j \ne i} \eta_{j\bar{j},p\bar{p}} \Big)\\
    &\quad  - \frac{1}{n - 1}\sum_{a \ne i, b \ne i, a \ne b} u_{a\bar{b}} \eta_{b\bar{a},p\bar{p}}.
   \end{split}
 \end{equation*}
 Note that under the normal coordinate system, the curvature $(R_{i\bar{j}k\bar{l}})$ of $\eta$ reads
 \[
    R_{i\bar{j}k\bar{l}} = - \eta_{i\bar{j},k\bar{l}} + \sum_{a,b} \eta^{a\bar{b}} \eta_{i\bar{b},k} \eta_{a\bar{j},\bar{l}} = - \eta_{i\bar{j},k\bar{l}}, \quad \textup{at $P$}.
 \]
 This together with \eqref{eq:odPsiu} imply that
 \begin{equation} \label{eq:nmd2Psi}
   \begin{split}
    (\Psi_u)_{i\bar{i},p\bar{p}}
    &  = \Psi_{i\bar{i},p\bar{p}} + \frac{1}{n-1}\sum_{k \ne i} u_{k\bar{k}p\bar{p}}  - \frac{1}{n - 1}\sum_{k \ne i} u_{k\bar{k}}\Big(\sum_{j \ne k, j \ne i} R_{j\bar{j}p\bar{p}} \Big)\\
    &\quad  - \sum_{a \ne i, b \ne i, a \ne b} \Psi_{a\bar{b}} R_{a\bar{b}p\bar{p}}.
   \end{split}
 \end{equation}

 We compute at $P$ that
 \begin{align*}
    L\phi = & (n-1)\sum_l (\Psi_u)^{l\bar{l}} \big(\frac{\sqrt{-1}}{2}\partial \bar{\partial} \phi \wedge \eta^{n-2}\big)_{l\bar{l}}
       = \sum_l \sum_{p \ne l} (\Psi_u)^{l\bar{l}} \phi_{p\bar{p}}.
 \end{align*}
 Note that
 \begin{equation} \label{eq:1stphi}
     0 = \phi_p (P) = h_p + (\Delta_{\eta} u)_p - A u_p.
 \end{equation}
 Apply \eqref{eq:1stphi} to obtain that
 \begin{align*}
    0 \ge \phi_{p\bar{p}} (P) = h_{p\bar{p}} + (\Delta_{\eta} u)_{p\bar{p}} - A
    u_{p\bar{p}}.
 \end{align*}
 It follows that
 \begin{equation}\lab{3333}
 \begin{split}
   0
   & \ge L \phi = \sum_l \sum_{p \ne l} (\Psi_u)^{l\bar{l}} \phi_{p\bar{p}} \\
   & = \sum_l \sum_{p \ne l} (\Psi_u)^{l\bar{l}}[h_{p\bar{p}} + (\Delta_{\eta} u)_{p\bar{p}}]
   - A \sum_l \sum_{p \ne l} (\Psi_u)^{l\bar{l}} u_{p\bar{p}}.
 \end{split}
 \end{equation}
Notice that
   \begin{align}
    & \sum_l \sum_{p \ne l} (\Psi_u)^{l\bar{l}}[h_{p\bar{p}} + (\Delta_{\eta} u)_{p\bar{p}}] \notag \\
    & = \sum_l \sum_{p \ne l} (\Psi_u)^{l\bar{l}} h_{p\bar{p}} + \sum_{l,a} \sum_{p \ne l} (\Psi_u)^{l\bar{l}} u_{a\bar{a}p\bar{p}}
    +  \sum_l \sum_{p \ne l} (\Psi_u)^{l\bar{l}} \sum_{a,b} \eta^{a\bar{b}}_{,p\bar{p}} u_{a\bar{b}} \notag \\
    & = \sum_l \sum_{p \ne l} (\Psi_u)^{l\bar{l}} h_{p\bar{p}} + \sum_{l,a} \sum_{p \ne l} (\Psi_u)^{l\bar{l}} u_{a\bar{a}p\bar{p}}
    +  \sum_{l,a} \sum_{p \ne l} (\Psi_u)^{l\bar{l}} R_{a\bar{a}p\bar{p}} u_{a\bar{a}} \label{eq:4thtexp} \\
    & \quad - (n-1) \sum_l \sum_{p \ne l}\sum_{a \ne b} (\Psi_u)^{l\bar{l}}R_{b\bar{a}p\bar{p}} \Psi_{b\bar{a}}, \qquad \textup{\Big(by \eqref{eq:odPsiu}\Big).} \notag
 \end{align}

 Here the fourth derivative term can be handled by the equation \eqref{eq:Cf}:
 We rewrite \eqref{eq:Cf} as
\[
        \log \det\Psi_u = F + \log \det\Psi.
\]
Differentiating this in the direction of $\p/\p z^a$ yields
\[
    \sum_{k,l} (\Psi_u)^{k\bar{l}}(\Psi_u)_{k\bar{l},a} = (F + \log\det\Psi)_{a}.
\]
and then,
\begin{eqnarray*}
 \sum_{k,l}(\Psi_u)^{k\bar{l}}(\Psi_u)_{k\bar{l},a\bar{b}}=(F + \log\det\Psi)_{a\bar{b}}
+\sum_{k,l,p,q}(\Psi_u)^{k\bar{q}}(\Psi_u)^{p\bar{l}}(\Psi_u)_{k\bar{l},a} (\Psi_u)_{p\bar{q},\bar{b}}.
\end{eqnarray*}
Contracting this with $(\eta^{a\bar{b}})$ and applying the normal coordinates yield that
\[
   \sum_{l,a}(\Psi_u)^{l\bar{l}}(\Psi_u)_{l\bar{l},a\bar{a}}=\sum_a (F + \log\det\Psi)_{a\bar{a}}
+\sum_{k,l,a}\frac{\big|(\Psi_u)_{k\bar{l},a}\big|^2}{(\Psi_u)_{l\bar{l}}(\Psi_u)_{k\bar{k}}}.
\]
This together with \eqref{eq:nmd2Psi} imply that
\begin{align*}
   & \sum_{l,a} (\Psi_u)^{l\bar{l}}\Psi_{l\bar{l},a\bar{a}} + \frac{1}{n-1}\sum_{l,a} \sum_{p \ne l} (\Psi_u)^{l\bar{l}} u_{p\bar{p}a\bar{a}} \\
   & \ge \sum_{k,l,a}\frac{\big|(\Psi_u)_{k\bar{l},a}\big|^2}{(\Psi_u)_{l\bar{l}}(\Psi_u)_{k\bar{k}}} + \frac{1}{n-1} \sum_{l,a} (\Psi_u)^{l\bar{l}}\sum_{p \ne l} u_{p\bar{p}}\Big(\sum_{m \ne p, m \ne l}R_{m\bar{m}a\bar{a}} \Big) \\
   & + \sum_{l,a} (\Psi_u)^{l\bar{l}} \Big(\sum_{p \ne l, q \ne l, p \ne q} \Psi_{p\bar{q}} R_{p\bar{q}a\bar{a}} \Big)
   + \Delta_{\eta} F + \Delta_\eta(\log \det \Psi).
\end{align*}
Combining this with \eqref{eq:4thtexp} yield
\begin{equation} \label{eq:curPsiu}
 \begin{split}
   & \sum_l \sum_{p \ne l} (\Psi_u)^{l\bar{l}}(h_{p\bar{p}} + (\Delta_{\eta} u)_{p\bar{p}})  \\
   & \ge \sum_{l,a} \sum_{p \ne l} (\Psi_u)^{l\bar{l}}R_{a\bar{a}p\bar{p}} u_{a\bar{a}} + \sum_{l,a} \sum_{p \ne l} (\Psi_u)^{l\bar{l}}u_{p\bar{p}}\Big(\sum_{m \ne p, m \ne l} R_{a\bar{a}m\bar{m}} \Big) \\
   & \quad + (n-1)\sum_{k,l,a}\frac{\big|(\Psi_u)_{k\bar{l},a}\big|^2}{(\Psi_u)_{l\bar{l}}(\Psi_u)_{k\bar{k}}} +(n-1) \Delta_{\eta} F + (n-1)\Delta_{\eta}(\log \det \Psi) \\
   & \quad + \sum_l \sum_{p \ne l} (\Psi_u)^{l\bar{l}} h_{p\bar{p}} - (n - 1) \sum_{l, a} (\Psi_u)^{l\bar{l}} \Psi_{l\bar{l},a\bar{a}}   \\
   & \quad + (n-1)\sum_{l,a}(\Psi_u)^{l\bar{l}} \Big(\sum_{p \ne l, q \ne l, p \ne q}\Psi_{p\bar{q}} R_{p\bar{q}a\bar{a}} \Big)\\
   & \quad - (n-1)\sum_l \sum_{p \ne l} \sum_{a \ne b} (\Psi_u)^{l\bar{l}} R_{a\bar{b}p\bar{p}} \Psi_{a\bar{b}}.
 \end{split}
\end{equation}

The first two terms on the right hand side of above inequality can be  handled as follows.
  \begin{equation}\lab{1111}
  \begin{split}
     & \sum_{l,a} \sum_{p \ne l} (\Psi_u)^{l\bar{l}}R_{a\bar{a}p\bar{p}} u_{a\bar{a}} + \sum_{l,a} \sum_{p \ne l} (\Psi_u)^{l\bar{l}}u_{p\bar{p}}\Big(\sum_{m \ne p, m \ne l} R_{a\bar{a}m\bar{m}} \Big) \\
     & = \sum_{l,a} (\Psi_u)^{l\bar{l}} R_{l\bar{l}a\bar{a}} u_{l\bar{l}}  - \sum_{l,a} (\Psi_u)^{l\bar{l}} R_{a\bar{a}l\bar{l}} u_{a\bar{a}} + \sum_{l,p} \sum_{a \ne l} (\Psi_u)^{l\bar{l}} u_{a\bar{a}} R_{a\bar{a}p\bar{p}} \\
     & \quad + \sum_{l,a} \sum_{p \ne l} (\Psi_u)^{l\bar{l}}u_{p\bar{p}}\Big(\sum_{m \ne p, m \ne l} R_{a\bar{a}m\bar{m}} \Big) \\
     & = \frac{1}{2} \sum_{l,a} (\Psi_u)^{l\bar{l}} R_{l\bar{l}a\bar{a}} (u_{l\bar{l}} - u_{a\bar{a}}) + \frac{1}{2} \sum_{l,a} (\Psi_u)^{a\bar{a}} R_{l\bar{l}a\bar{a}} (u_{a\bar{a}} - u_{l\bar{l}}) \\
     & \quad + (n-1)\sum_l \big(\sum_{m \ne l} R_{m\bar{m}}\big) (\Psi_u)^{l\bar{l}}  \big[(\Psi_u)_{l\bar{l}} - \Psi_{l\bar{l}}\big] \qquad \textup{\Big(by \eqref{eq:diagPsiu}\Big)} \\
     & = \frac{1}{2}\sum_{l,a} R_{l\bar{l}a\bar{a}} \frac{(u_{l\bar{l}} - u_{a\bar{a}})[(\Psi_u)_{a\bar{a}} - (\Psi_u)_{l\bar{l}}]}{(\Psi_u)_{l\bar{l}} (\Psi_u)_{a\bar{a}}} \\
     & \quad + (n - 1)^2 \sum_l R_{l\bar{l}}
     - (n -1) \sum_l (\Psi_u)^{l\bar{l}} \Psi_{l\bar{l}} \big(\sum_{m \ne l} R_{m\bar{m}} \big).
     \end{split}
  \end{equation}
Apply \eqref{eq:diagPsiu} to estimate the first term of last equality
   \begin{equation}\lab{2222}
   \begin{split}
     & \frac{1}{2}\sum_{l,a} R_{l\bar{l}a\bar{a}}\frac{(u_{l\bar{l}} - u_{a\bar{a}})[(\Psi_u)_{a\bar{a}} - (\Psi_u)_{l\bar{l}}]}{(\Psi_u)_{l\bar{l}} (\Psi_u)_{a\bar{a}}} \\
     & = \frac{n - 1}{2}\sum_{l,a} R_{l\bar{l}a\bar{a}}\frac{[(\Psi_u)_{a\bar{a}} - (\Psi_u)_{l\bar{l}}]^2}{(\Psi_u)_{l\bar{l}} (\Psi_u)_{a\bar{a}}} \\
    & \quad + \frac{n - 1}{2} \sum_{l,a}R_{l\bar{l}a\bar{a}}\frac{(\Psi_{l\bar{l}} - \Psi_{a\bar{a}})[(\Psi_u)_{a\bar{a}} - (\Psi_u)_{l\bar{l}}]}{(\Psi_u)_{l\bar{l}} (\Psi_u)_{a\bar{a}}} \\
    & \ge (n - 1) \sum_{l,a}R_{l\bar{l}a\bar{a}}\frac{\Psi_{l\bar{l}} - \Psi_{a\bar{a}}}{(\Psi_u)_{l\bar{l}}}, \qquad \textup{by \eqref{eq:nonbis}.}
   \end{split}
   \end{equation}
Combining (\ref{eq:curPsiu}) with (\ref{1111}) and then with (\ref{2222}), we obtain
\begin{equation} \label{eq:Lphi1}
  \begin{split}
   & \sum_l \sum_{p \ne l} (\Psi_u)^{l\bar{l}}(h_{p\bar{p}} + (\Delta_{\eta} u)_{p\bar{p}})  \\
   & \ge - C_1(n-1) \sum_l (\Psi_u)^{l\bar{l}} - (n-1)^2 C_1 + (n-1)\inf \Delta_{\eta} F .
  \end{split}
\end{equation}
Here and throughout this section, we denote by $C_1>0$ a generic constant depending only on $\Psi$ and the curvature of $\eta$.

Substituting \eqref{eq:Lphi1} into \eqref{3333} yields
\begin{align*}
  0
  \ge L \phi
  & \ge  - A \sum_l \sum_{p \ne l} (\Psi_u)^{l\bar{l}} u_{p\bar{p}} - C_1(n - 1) \sum_l (\Psi_u)^{l\bar{l}} \\
  & \quad - (n-1)^2 C_1 - \inf \Delta_{\eta} F \\
  & = - n (n-1) A + (n-1) A \sum_l (\Psi_u)^{l\bar{l}} \Psi_{l\bar{l}} - C_1 (n-1) \sum_l (\Psi_u)^{l\bar{l}} \\
  & \quad  - (n-1)^2 C_1 + (n-1) \inf \Delta_{\eta} F.
\end{align*}
Now we choose $A>0$ sufficiently large so that
\[
   A \inf_X (\min_l \Psi_{l\bar{l}}) \ge 2 C_1.
\]
It follows that
\begin{align*}
& \frac{nA}{C_1}+(n - 1) - \frac{\inf \Delta_{\eta} F}{C_1}
 \ge  \sum_{l=1}^n (\Psi_u)^{l\bar{l}}  \\
 &\ge  \left[\sum_{i=1}^n(\Psi_u)_{i\bar{i}}\right]^{\frac{1}{n-1}}\left\{\det                \big[(\Psi_u)_{i\bar{j}}\big]\right\}^{\frac{-1}{n-1}}  \\
 &  =  \left[\sum_{i=1}^n (\Psi_u)_{i\bar{i}}\right]^{\frac{1}{n-1}}e^{\frac{- F}{n-1}} (\det \Psi)^{\frac{-1}{n-1}}.
\end{align*}
Hence,
\[
   h + \Delta_{\eta} u = \sum_{i=1}^n (\Psi_u)_{i\bar{i}} \le C_2 \qquad \textup{at $P$}.
\]
Here and throughout this section, we denote by $C_2$ a generic positive constant depending only on $n$, $\Psi$, $\eta$, $\sup \Delta_{\eta} F$, and $\sup F$. Therefore, at any point in $X$,
\[
   (h + \Delta_{\eta} u) \le (h + \Delta_{\eta} u)(P) + A u - A u(P) \le C_2 + C_2 \big( u - \inf_X
   u\big).
\]
Since $[(\Psi_u)_{i\bar{j}}]$ is positive definite everywhere, we have
\[
   |(\Psi_u)_{i\bar{j}}| \le C_2 + C_2 (u - \inf_X u), \qquad \textup{for all $1 \le i, j \le n$}.
\]
This completes the proof.
\qed
\end{pflemC2}
Lemma~\ref{le:C2} enables us to establish the $C^2$ estimate for equation~\eqref{eq:FT}:
\begin{coro} \label{co:C2}
  For any $f \in C^{\infty}(X)$, let $u \in C^{\infty}(X)$ be a solution of
  \begin{equation} \label{eq:openeqn}
     \frac{\det (\omega_u^{n-1})}{\det (\omega_0^{n-1})} = e^{(n-1)f} \left( \frac{\int_X \omega_u^n}{\int_X e^f \omega_0^n} \right)^{n-1} ,
  \end{equation}
  where $\omega_u$ is a positive $(1,1)$--form on $X$ such that
  \[
     \omega_u^{n-1} = \omega_0^{n-1} + \ppr u\wedge \eta^{n-2} > 0.
  \]
  Then, we have
  \begin{equation} \label{eq:Lap}
     \Delta_{\eta} u \le C + C (u - \inf_X u) \qquad \textup{on $X$},
  \end{equation}
  and
  \[
      \sup_X |\omega_u^{n-1}|_{\eta} \le C + C ( \sup_X u - \inf_X u),
  \]
  where $C> 0$ is a constant depending only on $f$, $\eta$, $n$, and $\omega_0$.
\end{coro}
\begin{proof}
  Let
  \[
     F = (n-1) \left( f + \log \int_X \omega_u^n - \log \int_X e^f \omega_0^n \right).
  \]
  To apply Lemma~\ref{le:C2}, it suffices to estimate $\inf (\Delta_{\eta} F)$ and $\sup F$. Note that
  \[
     \Delta_{\eta} F = (n-1) \Delta_{\eta} f.
  \]
  Applying the maximum principle to \eqref{eq:openeqn} at the points where $u$ attain its maximum and minimum, respectively,  yields a uniform bound for the constant:
  \[
     - \sup f \le \log \int_X \omega_u^n - \log \int_X e^f \omega_0^n  \le - \inf f.
  \]
  This implies that $\sup |F| \le (n-1) (\sup f - \inf f)$.
\end{proof}

\section{$C^0$ estimates} \label{se:C0}

In this section, we first would like to derive the following general $C^0$ estimate. This then combining Corollary~\ref{co:C2} will settle the $C^0$ estimate for manifolds of nonnegative orthogonal bisectional curvature.
\begin{lemm} \label{le:C0}
  Let $(X,\eta)$ be an arbitrary K\"ahler manifold with complex dimension $n \ge 2$. Suppose that $u \in C^2(X)$ satisfies
  \begin{align*}
     \Delta u & \le C_1 + C_1 (u - \inf_X u), \\
     \Delta u & > - C_2,
  \end{align*}
  where $\Delta$ stands for the Laplacian with respect to $\eta$, and $C_1, C_2$ are two positive constants.
  Then,
  \[
    \sup_X u - \inf_X u \le C,
  \]
  in which $C>0$ is a constant depending only on $\eta$, $n$, $C_1$,  and $C_2$.
\end{lemm}

 The proof uses Moser's iteration, consisting of the following two propositions. For simplicity, we denote, throughout this section, that
   \[
      \int h = \int_X h \, \eta^n, \qquad \textup{for all $h \in L^1(X,\eta)$},
   \]
   and for $p > 0$,
   \[
      \| h\|_p = \left(\int h^p \right)^{1/p}, \qquad \textup{for all $h \in L^p(X,\eta)$}.
   \]
   And we abbreviate $\Delta = \Delta_{\eta}$ in this section.
 \begin{prop} \label{pr:sup}
   Let $v \in C^2(X)$, $v > 0$ on $X$, satisfy that
   \begin{equation} \label{eq:sup}
      \Delta v + c v \ge d \qquad \textup{on $X$},
   \end{equation}
   where $c$ and $d$ are constants. Then, for any real number $p > 0$,
   \[
      \sup_X v \le C^{1/p} (1 + |c|)^{n/p} ( \| v\|_p + |d|),
   \]
   where $C>0$ is a constant depending only on $\eta$ and $n$.
 \end{prop}
 \begin{proof}
   Let
   \[
     \td{v} = v + |d|.
   \]
   Then,
   \[
     \td{v} \ge v > 0.
   \]
   Multiplying both sides of \eqref{eq:sup} by $- \td{v}^p$, $p \ge 1$, and then integrating by parts yield that
   \[
     p \int |\nabla \td{v}|^2 \td{v}^{p-1} \le c \int \td{v}^p v - d \int \td{v}^p.
   \]
   Then,
   \begin{align*}
      \int_X |\nabla \td{v}^{\frac{p+1}{2}} |^2
      & \le p |c| \int \td{v}^{p+1} + p |d| \int \td{v}^p \\
      & \le p(|c|+1) \int \td{v}^{p+1}, \qquad p \ge 1.
   \end{align*}

Now invoke the Sobolev inequality
   \[
      \|h\|_{2n/(n-1)}^2 \le C (\|\nabla h\|_2^2 + \|h \|_2^2), \qquad \textup{for all $h \in C^1(X)$}.
   \]
   Here and below, we denote by $C>0$ a generic constant depending only on $\eta$ and $n$.
   Substituting $h = \td{v}^{\frac{p+1}{2}}$ into the Sobolev inequality gives that
   \begin{equation} \label{eq:ita1}
      \| \td{v} \|_{(p+1)\kappa} \le [C (1 + |c|) (p+1)]^{\frac{1}{p+1}} \| \td{v} \|_{p+1}, \qquad \textup{for all $p \ge 1$.}
   \end{equation}
   Here
   \[
     \kappa = \frac{n}{n-1}.
   \]

   Now we fix a real number $p \ge 2$, and define a sequence $\{p_i\}$ as follows:
   \[
      p_0 = p, \qquad p_i = p_{i-1} \kappa = p \kappa^i, \quad \textup{for all $i = 1, 2, \ldots$}
   \]
   Iterating \eqref{eq:ita1} with respect to $\{p_i\}$ yields that
   \begin{align*}
     \| \td{v} \|_{p_{k+1}}
     & \le \exp\left(\log \big[C(1+|c|)\big] \sum_{i=0}^k \frac{1}{p_i} + \sum_{i=0}^k \frac{\log p_i}{p_i}\right) \| \td{v}\|_{p_0}\\
     & \le [C (1 + |c|)]^{n/p} \|\td{v}\|_{p}, \quad \textup{for any $k \ge 0$},
   \end{align*}
   where we use the fact that
   \[
      \sum_{i=0}^{\infty} \frac{1}{\kappa^i} = n.
   \]

   Letting $k$ tend to infinity gives that for $p\geq 2$,
   \begin{equation}\lab{4444}
    \sup_X \td{v} \le [C(1 + |c|)]^{n/p} \| \td{v} \|_p.
   \end{equation}
   For $0 < p < 2$, it follows from above inequality that
   \begin{align*}
      \sup \td{v} & \le [C(1 + |c|)]^{n/2} \left(\int \td{v}^2 \right)^{1/2} \\
      & \le [C(1 + |c|)]^{n/2} \left(\int \td{v}^p \right)^{1/2} (\sup \td{v})^{(1 - p/2)}.
   \end{align*}
   Then we still have (\ref{4444}). So for any real number $p>0$,
   $$\sup_X v\leq \sup_X\tilde v\leq C^{1/p}(1+|c|)^{n/p}(\|v\|_p+|d|).$$
 \end{proof}

 \begin{prop} \label{pr:inf}
   Let $v \in C^2(X)$, $v > 0$ on $X$ and satisfy
   \begin{equation}\label{eq:inf}
      \Delta v - c v \le 0 \qquad \textup{on $X$},
   \end{equation}
   where $c$ is a constant. Then, there exists a real number $p_0  > 0$, depending on $\eta$, $n$, and $c$, such that
   \[
      \inf_X v \ge C^{-1/p_0} (1 + |c|)^{-n/p_0} \| v \|_{p_0},
   \]
   where $C> 0$ depends only on $\eta$ and $n$.
 \end{prop}
 \begin{proof}
    By \eqref{eq:inf} we have
    \[
       \Delta (v^{-1}) + c v^{-1} \ge - v^{-2} (\Delta v - c v) \ge 0.
    \]
    Applying Proposition~\ref{pr:sup} to $v^{-1}$ yields that, for any $p>0$,
    \[
       \sup (v^{-1}) \le C^{1/p} (1 + |c|)^{n/p} \| v^{-1} \|_p.
    \]
    It follows that
    \[
      \inf v \ge C^{-1/p} (1 + |c|)^{-n/p} \| v\|_p \left( \int v^{-p} \cdot \int v^p \right)^{-1/p}.
    \]
    Then, it suffices to show that, there exists some $p_0 > 0$ such that
    \begin{equation} \label{eq:inf2}
        \int v^{-p_0} \cdot \int v^{p_0} \le C.
    \end{equation}
    Here and below, we always denote by $C>0$ a generic constant depends only on $\eta$ and $n$, unless otherwise indicated.
    Denote
    \[
       w = \log v - \frac{\int \log v}{ \int \eta^n}.
    \]
    To show \eqref{eq:inf2}, it is sufficient to establish
    \begin{equation} \label{eq:inf3}
       \int e^{p_0 |w|} \le C,
    \end{equation}
    for, \eqref{eq:inf3} implies both
        \[
      \int e^{p_0 w}  \le C, \qquad \textup{and $\quad \int e^{-p_0 w}  \le C$};
        \]
    and multiplying these two inequalities gives \eqref{eq:inf2}.

    Note that
    \[
       e^{p_0|w|} = \sum_{m=0}^{\infty} \frac{p_0^m |w|^m}{m!}.
    \]
    Let us estimate $\|w\|_m$   for each $m \ge 1$. Multiplying both sides of \eqref{eq:inf} by $\phi v^{-1}$, where $\phi \in C^1(X)$ and $\phi \ge 0$, and then integrating by parts yield that
    \begin{equation} \label{eq:phi}
       \int \phi |\nabla w|^2 \le c \int \phi + \int \nabla \phi \cdot \nabla w.
    \end{equation}
    We first set $\phi \equiv 1$ in \eqref{eq:phi} to obtain that
    \begin{equation} \label{eq:gradw}
       \int |\nabla w|^2 \le |c| \int \eta^n .
    \end{equation}
    We can apply Poincar\'e inequality to get
    \begin{equation} \label{eq:L2}
       \int w^2 \le C |c|.
    \end{equation}
    Then, by H\"older inequality,
    \begin{equation} \label{eq:L1}
       \int |w| \le C |c|^{1/2}.
    \end{equation}

    It remains to estimate $\|w\|_m$ for $m \ge 3$. We now set $\phi = w^{2p}$, $p \ge 1$, in \eqref{eq:phi}. Then,
    \[
       \int |w|^{2p} |\nabla w|^2 \le |c| \int |w|^{2p} + 2p \int |w|^{2p -1} |\nabla w|^2.
    \]
    By Young's inequality,
    \[
       2p |w|^{2p-1} \le \frac{2p -1}{2p} |w|^{2p} + (2p)^{2p -1}.
    \]
    It follows that
    \[
      \int |w|^{2p} |\nabla w|^2 \le 2p |c| \int |w|^{2p} + (2p)^{2p} \int |\nabla w|^2.
    \]
    Observe that
    \[
       |\nabla w^p|^2 = p^2 w^{2p-2} |\nabla w|^2 \le w^{2p} |\nabla w|^2 + p^{2p} |\nabla w|^2.
    \]
    We then have
    \[
       \int |\nabla w^p|^2 \le 2p|c| \int |w|^{2p} + 2 (2p)^{2p} \int |\nabla w|^2, \qquad \textup{for all $p \ge 1$}.
    \]
    Apply the Sobolev inequality to obtain that
    \begin{align*}
       \left( \int w^{2p\kappa} \right)^{\kappa}
       & \le C p (1 + |c|) \int |w|^{2p} + C (2p)^{2p} \int |\nabla w|^2\\
       & \le C p (1 + |c|) \left( \int |w|^{2p} + (2p)^{2p}\right), \qquad \textup{by \eqref{eq:gradw}.}
    \end{align*}
    Here we denote
    \[
        \kappa = \frac{n}{n-1}.
    \]
    Hence, we have for all $p \ge 1$ that
    \begin{equation} \label{eq:infita}
        \|w \|_{2p\kappa} \le [C(1+|c|)]^{\frac{1}{2p}} (2p)^{\frac{1}{2p}} \big( \|w\|_{2p} + 2p \big),
    \end{equation}
    in view of the inequality
    \[
       (a + b)^{\epsilon} \le a^{\epsilon} + b^{\epsilon}, \qquad \textup{for all $0< \epsilon <1$, $a \ge 0$, and $b\ge0$}.
    \]
    We shall iterate \eqref{eq:infita} with respect to the sequence $\{p_i\}_{i=0}^{\infty}$ given below:
    \[
       p_0 = 2, \qquad p_{i} = p_{i-1} \kappa = 2 \kappa^i \quad \textup{for all $i \ge 1$}.
    \]
    Thus, we obtain for each $k \ge 0$ that
    \begin{align*}
      \|w \|_{p_{k+1}}
      & \le C \sum_{i=0}^k p_i + \exp\left(\log [C(1+|c|)] \sum_{i=0}^k \frac{1}{p_i} + \sum_{i=0}^k \frac{\log p_i}{p_i} \right) \|w\|_{p_0} \\
      & \le C p_k + C ( 1 + |c|)^{n/2} \| w \|_2,
    \end{align*}
    in which we use the fact that
    \[
       \sum_{i=0}^k \kappa^i \le n \kappa^k.
    \]
 Now note that for any integer $m \ge 2$, there exists an integer $i \ge 0$ such that
 \[
    2 \kappa^i \le m < 2 \kappa^{i+1}.
 \]
 Then,
 \begin{align*}
    \|w\|_{m}
    & \le \|w\|_{p_{i+1}} \le C m + C (1 + |c|)^{n/2} \|w\|_2 \\
    & \le C m + C (1 + |c|)^{n/2} \qquad \textup{\Big(by \eqref{eq:L2}\Big)} \\
    & \le C ( 1 + |c|)^{n/2} m.
 \end{align*}
 Hence,
 \[
    \int \frac{|w|^m}{m!} \le C^m (1 + |c|)^{\frac{nm}{2}} \frac{m^m}{m!} \le C^m (1 + |c|)^{\frac{nm}{2}} e^m.
 \]
 Let
 \[
    p_0 = \frac{1}{2C(1 + |c|)^{n/2} e};
 \]
 and then
 \[
    \int \frac{p_0^m|w|^m}{m!} \le \frac{1}{2^m}, \qquad \textup{for all $m \ge 2$}.
 \]
 This together with \eqref{eq:L1} yields \eqref{eq:inf3}. This completes the proof.
\end{proof}

We are in a position to prove Lemma~\ref{le:C0}.
\begin{pflemC0}
  Let
  \[
     v = u - \inf_X u + 1.
  \]
  Then,
  \[
     v \ge 1, \qquad \textup{and \quad $\inf_X v = 1$},
  \]
  since $X$ is compact and so $u$ attains its minimum.
  On the other hand, we have
  \begin{equation} \label{eq:C0inf}
     \Delta v - C_1 v \le 0, 
  \end{equation}
  and
  \begin{equation} \label{eq:C0sup}
     \Delta v > - C_2.
  \end{equation}
  Applying Proposition~\ref{pr:inf} to \eqref{eq:C0inf} obtains that
  \[
     \inf_X v \ge C^{-1/p_0} ( 1 + |C_1|)^{-n/p_0} \|v\|_{p_0}.
  \]
  Here $p_0 > 0$ is a number depending only on $\eta$, $n$, and $C_1$; and $C > 0$ is a constant depending only on $\eta$ and $n$.
  Applying Proposition~\ref{pr:sup} to \eqref{eq:C0sup} with $p = p_0$ yields that
  \[
     \sup_X v \le (C')^{1/p_0} ( \| v\|_{p_0} + C_2),
  \]
  where $C'>0$ depends only on $\eta$ and $n$. Combining these two inequalities we have
  \begin{align*}
     \sup_X v
     & \le (C')^{1/p_0} \left[C^{1/p_0} ( 1 + |C_1|)^{n/p_0} \inf_X v + C_2 \right] \\
     & = (C')^{1/p_0} \left[C^{1/p_0} ( 1 + |C_1|)^{n/p_0} + C_2 \right].
  \end{align*}
  It follows that
  \[
    \sup_X u - \inf_X u \le \sup_X v \le C,
  \]
  where $C>0$ depends only on $\eta$, $n$, $C_1$, and $C_2$.
  \qed
\end{pflemC0}

Let us now return to equation~\eqref{eq:FT}. We let $(X, \eta)$ be the complex $n$-dimensional K\"ahler manifold of nonnegative quadratic bisectional curvature, and $\omega_0$ be a Hermitian metric on $X$.
\begin{coro} \label{co:C0}
  Given any $f \in C^{\infty}(X)$, let $u \in C^{\infty}(X)$ be a solution of
  \begin{equation*}
     \frac{\det (\omega_u^{n-1})}{\det (\omega_0^{n-1})} = e^{(n-1)f} \left( \frac{\int_X \omega_u^n}{\int_X e^f \omega_0^n} \right)^{n-1} ,
  \end{equation*}
  where $\omega_u$ is a positive $(1,1)$--form such that
  \[
     \omega_u^{n-1} = \omega_0^{n-1} + \ppr u\wedge \eta^{n-2} > 0 \quad \textup{on $X$}.
  \]
  Then,
  \[
      \sup_X |\omega_u^{n-1}|_{\eta} \le C,
  \]
  where $C> 0$ is a constant depending only on $f$, $\eta$, $n$, and $\omega_0$.
\end{coro}
\begin{proof}
  By Corollary~\ref{co:C2}, it suffices to estimate $(\sup u - \inf u)$. Contracting
  \[
        \omega_0^{n-1} + \ppr u\wedge \eta^{n-2} > 0
  \]
  with $\eta$ yields that
  \[
      \Delta_{\eta} u > - \frac{n \eta \wedge \omega_0^{n-1}}{\eta^n} > - C_2 \qquad \textup{on $X$}.
  \]
  Here the constant $C_2 > 0$ depends only on $\eta$, $n$, and $\omega_0$. We have \eqref{eq:Lap}, on the other hand. Therefore, the result is an immediate consequence of Lemma~\ref{le:C0}.
\end{proof}

\section{H\"older estimates for second derivatives} \label{se:Hld}

Let $X$ be a $n$-dimensional K\"ahler manifold, $\eta$ be a K\"ahler metric on $X$, and $\omega_0$ be a balanced metric on $X$. We will establish the following estimate.
\begin{lemm} \label{le:Hld}
Given $F \in C^2(X)$, let $u \in C^4(X)$ satisfy that 
\[
   \omega_0^{n-1} + \ppr u \wedge \eta^{n-2} > 0 \qquad \textup{on $X$},
\]
and that
\begin{equation} \label{eq:det}
\det [\omega_0^{n-1}+\ppr u\wedge \eta^{n-2}]= e^{F} \det\omega_0^{n-1}.
\end{equation}
Suppose that
\begin{equation} \label{eq:trac}
   \sup_X |\omega_0^{n-1} + \sqrt{-1}/2\partial \bar{\partial} u \wedge \eta^{n-2}|_{\eta} \le C_3
\end{equation}
for some constant $C_3>0$. Then,
\[
   \| u \|_{C^{2,\alpha}(X)} \le C,
\]
where $0 < \alpha < 1$ and $C>0$ are constants depending only on $C_3$, $n$, $\omega_0$, and $\eta$.
\end{lemm}

We shall apply the Evans--Krylov theory (see, for example, Gilbarg--Trudinger~\cite[p. 461, Theorem 17.14]{GT}.), which is on the real fully nonlinear elliptic equation. Note that Evans--Krylov theory is based on the \emph{weak Harnack estimate} (see, for example, \cite[p. 246, Theorem 9.22]{GT}), which, in turn, makes uses of the Aleksandrov's maximum principle (see, for example, \cite[p. 222, Lemma 9.3]{GT}).

We first adapt the Aleksandrov's maximum principle to the complex setting. To see this, we start from the following result (see, for example, Lemma 9.2 in \cite{GT}): Let $\Omega \subset \mathbb{C}^n$ be a bounded domain with smooth boundary.
\begin{Aleks}
 For $v \in C^2(\overline{\Omega})$ with $v \le 0$ on $\partial \Omega$, we have
\begin{equation} \label{eq:Aleks}
   \sup_{\Omega} v \le \frac{\dm(\Omega)}{\sigma_{2n}^{1/(2n)}} \left( \int_{\Gamma^+_v} |\det D^2 v| \right)^{\frac{1}{2n}}.
\end{equation}
Here $\sigma_{2n}$ is the volume of unit ball in $\mathbb{C}^n$, $D^2 v$ denotes the real Hessian matrix of $v$, and $\Gamma^+_v$ is the \emph{upper contact set} of $v$, i.e.,
\[
   \Gamma_v^+ = \{ y \in \Omega;  v(x) \le v(y) + Dv(y)\cdot (x - y) \;\; \textup{for all $x \in \Omega$} \}.
\]
\end{Aleks}
Then, it suffices to control the real Hessian $D^2 v$ by the complex Hessian $(v_{i\bar{j}})$ of $v$, over $\Gamma_v^+$. Note that $\Gamma_v^+ \subset \{ y \in \Omega; (D^2 v) (y) \le 0 \}$. We shall make use of the following inequality:
\begin{prop} \label{pr:cxreal}
   Let $w$ be a real $C^2$ function in $\Omega$. For $P \in \Omega$ such that $D^2 w \ge 0$,
   \[
       \det (D^2 w) \le 8^n |\det w_{i\bar{j}} |^2 \qquad \textup{at $P$}.
   \]
\end{prop}
\begin{proof}
   Recall that
   \[
      \frac{\p}{\p z^i} = \frac{1}{2}\left(\frac{\p}{\p x^i} - \sqrt{-1} \frac{\p}{\p y^i} \right), \qquad 1 \le i \le n.
   \]
   We denote
   \[
      w_{x^i} = \frac{\p w}{\p x^i}, \quad w_{x^i y^j} = \frac{\p^2 w}{\p x^i \p y^j}, \quad \ldots .
   \]
   Then,
   \[
      w_{i\bar{j}} = \frac{1}{4} \left( w_{x^i x^j} + w_{y^i y^j} \right) + \frac{\sqrt{-1}}{4} \left( w_{x^i y^j} - w_{x^j y^i} \right), \quad 1 \le i, j \le n.
   \]
   Since $D^2 w \ge 0$ at $P$, we can choose a coordinate system $(x^1,y^1, \ldots, x^n, y^n)$ near $P$ such that $D^2 w$ is diagonalized at $P$; and hence,
   \[
      w_{x^i x^i} \ge 0, \qquad w_{y^i y^i} \ge 0, \qquad \textup{for all $1 \le i \le n$}.
   \]
   Then, under this coordinate system, the complex Hessian of $w$ is also diagonalized, i.e.,
   \[
      w_{i\bar{j}} = \frac{\delta_{ij}}{4} \left( w_{x^i x^i} + w_{y^i y^i} \right).
   \]
   It follows that, at $P$,
   \begin{align*}
      16^n |\det w_{i\bar{j}} |^2
      & = \prod_{i=1}^n  \left( w_{x^i x^i} + w_{y^i y^i} \right)^2 \\
      & \ge 2^n \prod_{i=1}^n w_{x^i x^i} \prod_{i=1}^n w_{y^i y^i} \\
      & = 2^n \det (D^2 w).
   \end{align*}
\end{proof}

Moreover, for any Hermitian matrix $(a^{i\bar{j}}) > 0$ on $\Gamma_v^+$, we have by the elementary inequality that
\begin{equation} \label{eq:eineq}
    \det (a^{i\bar{j}}) \det (- v_{i\bar{j}}) \le \left( \frac{- \sum_{i,j} a^{i\bar{j}} v_{i\bar{j}} }{n} \right)^n.
\end{equation}
Now apply Proposition~\ref{pr:cxreal} and \eqref{eq:eineq} to \eqref{eq:Aleks} to obtain the following complex version Aleksandrov's maximum principle (compare with \cite[p. 222, Lemma 9.3]{GT}):
\begin{lemm} \label{pr:Aleks2}
Let $(a^{i\bar{j}})$ be a positive definite Hermitian matrix in $\Omega$. For $v \in C^2(\overline{\Omega})$ with $v \le 0$ on $\partial \Omega$,
\[
   \sup_{\Omega} v \le \frac{2^n \dm(\Omega)}{ n\,\sigma_{2n}^{1/(2n)}} \left[\int_{\Gamma_v^+} \Big|\frac{- \sum a^{i\bar{j}} v_{i\bar{j}} }{\det (a_{i\bar{j}})^{1/n}} \Big|^{2n} \right]^{\frac{1}{2n}}.
\]
\end{lemm}
Then, the weak Harnack inequality below (compare with \cite[p. 246, Theorem~9.22]{GT}) follows from Lemma~\ref{pr:Aleks2} and the cube decomposition procedure.
\begin{KS}
        Let $v \in W^{2,2n}(\Omega)$ satisfy $\sum a^{i\bar{j}} v_{i\bar{j}} \le g$ in $\Omega$, where $g \in L^{2n}(\Omega)$, and $(a^{i\bar{j}})$ satisfies that
        \[
           0 < \lambda |\zeta|^2 \le \sum_{i,j} a^{i\bar{j}}(z) \zeta_i \zeta_j \le \Lambda |\zeta|^2, \qquad \textup{for all $z \in \Omega$ and $\zeta \in \mathbb{C}^n$},
        \]
in which $\lambda$ and $\Lambda$ are two constants. Suppose that $v \ge 0$ in an open ball $B_{2R}(y) \subset \Omega$ centered at $y$ of radius $2R$. Then,
\[
   \left(\frac{1}{|B_R|} \int_{B_R} v^p \right)^{1/p} \le C \left[\inf_{B_R} v + \frac{R}{\lambda} \| g \|_{L^{2n}(B_{2R})} \right],
\]
where $|B_R|$ denotes the measure of $B_R$, and $p > 0$ and $C > 0$ are constants depending only on $n$, $\lambda$, and $\Lambda$.
\end{KS}

 Let us denote by
\[
   E[(u_{i\bar{j}})] = \log \det \big[\omega_0^{n-1} + \ppr u \wedge \eta^{n-2} \big].
\]
To apply Evans--Krylov theory, it remains to check the following two conditions (\cite[p. 456]{GT}):
\begin{enumerate}
  \item  \label{en:EK1} $E$ is uniformly elliptic with respect to $(u_{i\bar{j}})$,
  \item  \label{en:EK2} $E$ is concave on the range of $(u_{i\bar{j}})$.
\end{enumerate}
As in Section~\ref{se:C2}, we denote $\Psi = \omega_0^{n-1}$ and
\begin{equation} \label{eq:defPsi}
   \Psi_u = \Psi + \ppr u \wedge \eta^{n-2}.
\end{equation}
We use the index convention \eqref{eq:index} for an $(n-1, n-1)$--form. Then,
\[
        E[ (u_{i\bar{j}}) ] = \log \det [(\Psi_u)_{i\bar{j}}];
\]
and thus,
\[
    \frac{\partial E}{\partial (\Psi_u)_{i\bar{j}}} = (\Psi_u)^{i\bar{j}}, \quad \frac{\partial^2 E}{\partial (\Psi_u)_{i\bar{j}} \partial (\Psi_u)_{k\bar{l}}} = - (\Psi_u)^{i\bar{l}} (\Psi_u)^{k\bar{j}}.
\]
Clearly, $E$ is concave on $[(\Psi_u)_{i\bar{j}}]$.
By \eqref{eq:det} and \eqref{eq:trac}, we know that the eigenvalues of $[(\Psi_u)_{i\bar{j}}]$ with respect to $(\eta_{i\bar{j}})$, have uniform bounds which depend only on $F$, $\omega_0$, and $C_3$. Therefore, $E$ is uniformly elliptic with respect to $[(\Psi_u)_{i\bar{j}}]$. Observe that by \eqref{eq:defPsi}, $[(\Psi_u)_{i\bar{j}}]$ depends linearly on $(u_{p\bar{q}})$. Since $(\eta_{k\bar{l}}) > 0$ on $X$, the conditions \eqref{en:EK1} and \eqref{en:EK2} follows immediately from the chain rule.

Now we can apply the procedure in \cite[p. 457--461]{GT}, and this proves Lemma~\ref{le:Hld}. As a corollary, we obtain the H\"older estimate of $C^2$ for equation \eqref{eq:FT}.
\begin{coro} \label{co:C3}
        Let $(X,\eta)$ an $n$-dimensional K\"ahler of nonnegative quadratic bisectional curvature, and $\omega_0$ be a Hermitian metric on $X$. Given any $f \in C^{\infty}(X)$, let $u \in C^{\infty}(X)$ be a solution of
  \begin{equation*}
     \frac{\det (\omega_u^{n-1})}{\det (\omega_0^{n-1})} = e^{(n-1)f} \left( \frac{\int_X \omega_u^n}{\int_X e^f \omega_0^n} \right)^{n-1} ,
  \end{equation*}
  where $\omega_u$ is a positive $(1,1)$--form such that
  \[
     \omega_u^{n-1} = \omega_0^{n-1} + \ppr u\wedge \eta^{n-2} > 0 \quad \textup{on $X$}.
  \]
  Then,
  \[
      \| u \|_{C^{2,\alpha}(X)} \le C,
  \]
  where $0 < \alpha < 1$ and $C> 0$ are constants depending only on $f$, $\eta$, $n$, and $\omega_0$.
\end{coro}

\section{Openness and uniqueness} \label{se:open}
Throughout this section, we let $\omega_0$ be a balanced metric, and let $\eta$ be an arbitrary K\"ahler metric, unless otherwise indicated.
We fix $k \ge n+4$, $0 < \alpha < 1$, and a function $f \in C^{k,\alpha}(X)$ satisfying
\[
    \int_X e^f \omega_0^n = V \equiv \int_X \omega_0^n .
\]
Here $C^{k,\alpha}(X)$ is the usual H\"older space on $X$.
Consider for $0 \le t \le 1$,
\begin{equation} \label{eq:CYt}
   \frac{\det (\omega_{u_t}^{n-1})}{\det (\omega_0^{n-1}) } = e^{(n-1)tf} \left(\frac{\int_X \omega_{u_t}^n}{\int_X e^{tf} \omega_0^n}  \right)^{n-1},
\end{equation}
where $u_t \in \mathcal{P}_{\eta}(\omega_0)$. By abuse of notation, in this section we denote
\[
   \mathcal{P}_{\eta} (\omega_0) = \{ v \in C^{k+2,\alpha}(X); \omega_0^{n-1} + \ppr v \wedge \eta^{n-2} > 0\}.
\]
 Let
\begin{equation} \label{eq:defT}
  \begin{split}
   T = \{ t\in [0,1] ; & \textup{ the equation \eqref{eq:CYt} has a solution $u_t \in C^{k+2,\alpha}(X)$ }  \\
   & \textup{such that $u_t \in \mathcal{P}_{\eta}(\omega_0)$. \}.}
  \end{split}
\end{equation}
Clearly, we have $0 \in T$.
\begin{lemm} \label{le:opent}
  Let $T$ be the set given as above. Then $T$ is open in $[0,1]$.
\end{lemm}
\begin{proof}
  Notice that \eqref{eq:CYt} is the same as
  \[
     \frac{\omega_{u_t}^n}{\omega_0^n} = e^{tf} \frac{\int_X \omega_{u_t}^n}{\int_X e^{tf} \omega_0^n}.
  \]
  As in Section~3 of \cite{FuWangWu}, we define
  \[
     M (w) \equiv \log \frac{\omega_{w}^n}{\omega^n_0} - \log \left(\frac{1}{V} \int_X \omega_{w}^n \right),
  \]
  for any $w \in \mathcal{P}_{\eta}(\omega_0)$. Then, $M(w) \in \mathcal{F}^{k,\alpha}(X)$,
  where $\mathcal{F}^{k,\alpha}(X)$ is the hypersurface in $C^{k,\alpha}(X)$ given by
  \[
   \mathcal{F}^{k,\alpha}(X) = \left\{ g \in C^{k,\alpha}(X); \int_X e^g \, \omega_0^n = V \right \}.
\]

  Now suppose that $t \in T$. Then, the corresponding $u_t$ defines a positive $(1,1)$--form $\omega_{u_t}$ such that
  \[
      \omega_{u_t}^{n-1} = \omega_0^{n-1} + \ppr u \wedge \eta^{n-2} > 0 \qquad \textup{on $X$};
  \]
  furthermore, $u_t$ satisfies that
  \[
      M(u_t) = tf + \log V - \log \left(\int_X e^{tf} \omega_0^n \right) \in \mathcal{F}^{k,\alpha}(X).
  \]
  The tangent space of $\mathcal{F}^{k,\alpha}(X)$ at $M(u_t)$ is identically the same as the Banach space $\mathcal{E}_t^{k,\alpha}(X)$, which consists of all $h \in C^{k,\alpha}(X)$ such that
\[
   \int_X h \, \omega_{u_t}^n = 0.
\]
  In view of the Implicit Function Theorem, it suffices to show that the linearization operator $L_t \equiv M_{u_t}$, given by
  \[
     L_t (v)
    = \frac{n(\sqrt{-1}/2)\p \bar{\p} v \wedge \eta^{n-2} \wedge \omega_{u_t}}{(n-1)\omega_{u_t}^n}
    - \frac{n \int_X \ppr v \wedge \eta^{n-2} \wedge \omega_{u_t}}{(n-1)\int_X \omega_{u_t}^n},
\]
is a linear isomorphism from $\mathcal{E}_t^{k+2,\alpha}(X)$ to $\mathcal{E}_t^{k,\alpha}(X)$. This is guaranteed by Lemma~13 in \cite{FuWangWu}. The proof is thus finished.
\end{proof}
\begin{rema}
  We thank John Loftin for pointing out that the openness argument in \cite{FuWangWu} also works for $\eta$ being a \emph{astheno-K\"ahler} metric, i.e., $\eta$ is a hermitian metric such that $\p \bar{\p} \eta^{n-2} = 0$.
\end{rema}
By the results in the previous section, we know that $T$ is also
closed, provided that the orthogonal bisectional curvature of $\eta$
is nonnegative. Therefore, the existence part in
Theorem~\ref{th:main} is proved. The uniqueness follows immediately
from the following proposition.
\begin{prop}
  Let $v \in \mathcal{P}_{\eta}(\omega_0)$ satisfying
  \begin{equation} \label{eq:uni}
      \det \big[\omega_0^{n-1} + \ppr v \wedge \eta^{n-2} \big] = \delta \det \omega_0^{n-1},
  \end{equation}
  where $\delta >0$ is a constant. Then, $v$ must be a constant function and $\delta = 1$.
\end{prop}
\begin{proof}
  Applying the maximum principle to equation~\eqref{eq:uni}
  at the maximum points of $v$ yields that $\delta \le 1$. Similarly, we get $\delta \ge 1$ by considering \eqref{eq:uni} at the minimum points of $v$. Thus, $\delta = 1$.
Then, we apply the arithmetic--geometric mean inequality to obtain
  \begin{equation*} 
    \begin{split}
    1 & = \left[\frac{\det \omega_v^{n-1}}{\det \omega_0^{n-1}}\right]^{1/n}
    \le 1 + \frac{1}{n} \sum_{i,j=1}^n (\omega_0^{n-1})^{i\bar{j}} \big(\ppr v \wedge \eta^{n-2}\big)_{i\bar{j}}\\
    & = 1 + \frac{\omega_0 \wedge \eta^{n-2} \wedge \ppr v }{\omega_0^n} \equiv 1 + K v.
    \end{split}
  \end{equation*}
Note that the linear operator $K$ so defined is uniformly elliptic, by the metric equivalence of $\eta$ and $\omega_0$ on the compact manifold $X$.
Applying the strong maximum principle to $K v \ge 0$ yields that $v$ is a constant function.
\end{proof}
Therefore, the proof of Theorem~\ref{th:main} is completed.

\end{document}